\documentclass[11pt]{article}

\usepackage[T1]{fontenc}
\usepackage{lmodern}
\usepackage[margin=1in]{geometry}
\usepackage{amsmath,amssymb,amsthm,mathtools}
\usepackage{tikz}
\usepackage{needspace}
\usepackage{microtype}
\usepackage[hidelinks]{hyperref}

\newtheorem{theorem}{Theorem}
\newtheorem{lemma}[theorem]{Lemma}
\theoremstyle{remark}
\newtheorem{remark}[theorem]{Remark}

\title{Congruent Triangular Faces, Reflections Allowed:\\
  Universal Realization and the Minimum Face Count in Problem B22}
\author{\small George M. Georgiou\thanks{\small School of Computer Science and Engineering, California State University, San Bernardino}\\[3pt]
\small \href{mailto:georgiou@csusb.edu}{georgiou@csusb.edu}}
\date{}

\begin{document}

\maketitle

\begin{abstract}
Problem B22 in \emph{Unsolved Problems in Geometry} asks which triangles occur
as the common face of a convex polyhedron, how many copies are needed, and how
they may be arranged.  We settle the existence and minimum-face-count questions
for triangles in the version that allows reflected copies; we do not classify
all attainable face counts, nor the possible arrangements.  Every
nondegenerate Euclidean triangle occurs: we
exhibit an explicit convex polyhedron, combinatorially an octahedron, all eight
of whose faces are congruent to a prescribed triangle.  We then determine the
minimum number of faces for \emph{every} triangle.  It is four for an acute
triangle; six for a right or obtuse isosceles triangle with side lengths
\((\lambda,\lambda,\beta)\) satisfying \(\lambda\sqrt2\leq\beta<\lambda\sqrt3\);
and eight in all remaining cases.
\end{abstract}

\section{The problem}

Problem B22 of Croft, Falconer, and Guy \cite{CFG} includes the following
question.

\medskip
\noindent
\emph{For which triangles \(T\) is there a convex polyhedron whose faces are
all congruent to \(T\), with or without allowing reflected copies?  When such
a polyhedron exists, how many copies are needed and how may they be arranged?}
\medskip

Problem B22 also asks the analogous question for congruent \(n\)-gons, which
we do not address.  This note treats the reflection-allowed version of the
triangular case.  Congruence below
therefore means ordinary Euclidean congruence, including orientation-reversing
isometries.  Remark~\ref{rem:chirality} shows that the construction of
Section~\ref{sec:eight} really does use both orientations, so it says nothing
about the orientation-preserving version of the problem.

\begin{theorem}[Universal realization]\label{thm:universal}
For every nondegenerate Euclidean triangle \(T\) there is a convex polyhedron
with exactly eight faces, each of them congruent to \(T\).
\end{theorem}

For a nondegenerate triangle \(T\), write \(m(T)\) for the least number of
faces of a convex polyhedron all of whose faces are congruent to \(T\).
Theorem~\ref{thm:universal} guarantees that \(m(T)\) is defined and that
\(m(T)\leq8\).

\begin{theorem}[The minimum face count]\label{thm:minimum}
Let \(T\) be a nondegenerate Euclidean triangle.  Then
\[
  m(T)=
  \begin{cases}
    4, & \text{if \(T\) is acute},\\[2pt]
    6, & \text{if \(T\) is isosceles with side lengths \((\lambda,\lambda,\beta)\)}\\
       & \text{\quad and \(\lambda\sqrt2\leq\beta<\lambda\sqrt3\)},\\[2pt]
    8, & \text{otherwise.}
  \end{cases}
\]
The three cases are exhaustive and mutually exclusive: an isosceles triangle
\((\lambda,\lambda,\beta)\) is acute exactly when \(\beta<\lambda\sqrt2\), and
the third case consists of the right and obtuse scalene triangles together
with the isosceles triangles satisfying \(\lambda\sqrt3\leq\beta<2\lambda\).
\end{theorem}

\section{An eight-face construction}\label{sec:eight}

Let the side lengths of the given triangle be ordered as
\[
  0<a\leq b\leq c,
  \qquad a+b>c,
\]
and put
\[
  M=\frac{b^{2}+c^{2}-a^{2}}{2}.
\]
Thus \(M=bc\cos\alpha\), where \(\alpha\) is the angle between the sides of
lengths \(b\) and \(c\).  Since \(a\) is the shortest side, \(\alpha\) is the
smallest angle of the triangle and is therefore acute, so
\begin{equation}\label{eq:Mrange}
  0<M<bc .
\end{equation}
Both inequalities are also immediate algebraically: \(b^{2}+c^{2}-a^{2}\geq
b^{2}>0\) because \(a\leq c\), while \(M<bc\) is equivalent to
\((c-b)^{2}<a^{2}\), which follows from \(c-b<a\).

Choose \(X\in(0,b^{2})\) so that
\begin{equation}\label{eq:X}
  (b^{2}-X)(c^{2}-X)=M^{2}.
\end{equation}
Such an \(X\) exists and is unique.  Writing \(f(X)=(b^{2}-X)(c^{2}-X)\), we
have \(f'(X)=2X-b^{2}-c^{2}\leq 0\) on \([0,b^{2}]\), with equality only at the
right endpoint and only when \(b=c\); so \(f\) is continuous and strictly
decreasing there, and it takes the values \(f(0)=b^{2}c^{2}>M^{2}\) and
\(f(b^{2})=0<M^{2}\) at the endpoints, the first by~\eqref{eq:Mrange}.
Explicitly,
\[
  X=\frac{b^{2}+c^{2}
  -\sqrt{(c^{2}-b^{2})^{2}+4M^{2}}}{2}.
\]
Define
\[
  r=\sqrt{X},\qquad
  d=\sqrt{b^{2}-X},\qquad
  s=\sqrt{c^{2}-X},
\]
so that \(r>0\), \(0<d\leq s\), and, by \eqref{eq:X}, \(ds=M\).  Finally set
\[
  h=\frac{s+d}{2},
  \qquad
  t=\frac{s-d}{2},
\]
so that \(h>0\), \(t\geq0\), \(h-t=d\), and \(h+t=s\).

Consider the following six points in \(\mathbb{R}^{3}\):
\begin{align*}
  U&=(0,0,h),          & V&=(0,0,-h),\\
  A&=(r,0,t),          & B&=(0,r,-t),\\
  C&=(-r,0,t),         & D&=(0,-r,-t).
\end{align*}
Geometrically, \(A,B,C,D\) lie on a circle of radius \(r\) about the
\(z\)-axis, alternating between the heights \(t\) and \(-t\), and \(U,V\) are
apices at heights \(\pm h\).  Let
\[
  P=\operatorname{conv}\{U,V,A,B,C,D\}.
\]
We claim that its facets are exactly
\begin{equation}\label{eq:facets}
  UAB,\ UBC,\ UCD,\ UDA,
  \qquad
  VAB,\ VBC,\ VCD,\ VDA;
\end{equation}
see Figure~\ref{fig:P}.

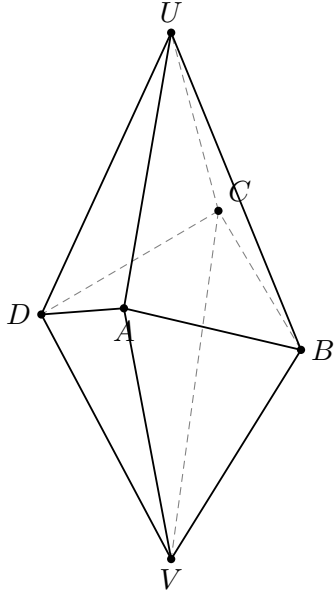
\begin{figure}[ht]
\centering
\begin{tikzpicture}[scale=1.55,
  vis/.style={line width=0.7pt},
  hid/.style={line width=0.4pt,densely dashed,gray},
  vx/.style={circle,fill=black,inner sep=1.1pt}]
  \coordinate (U) at (0.000,2.247);
  \coordinate (V) at (0.000,-2.247);
  \coordinate (A) at (-0.404,-0.106);
  \coordinate (B) at (1.110,-0.461);
  \coordinate (C) at (0.404,0.726);
  \coordinate (D) at (-1.110,-0.159);
  \draw[hid] (B) -- (C);
  \draw[hid] (C) -- (D);
  \draw[hid] (C) -- (U);
  \draw[hid] (C) -- (V);
  \draw[vis] (A) -- (B);
  \draw[vis] (A) -- (D);
  \draw[vis] (A) -- (U);
  \draw[vis] (A) -- (V);
  \draw[vis] (B) -- (U);
  \draw[vis] (B) -- (V);
  \draw[vis] (D) -- (U);
  \draw[vis] (D) -- (V);
  \foreach \p in {U,V,A,B,C,D} \node[vx] at (\p) {};
  \node[above]      at (U) {$U$};
  \node[below]      at (V) {$V$};
  \node[below,yshift=-1pt] at (A) {$A$};
  \node[right]      at (B) {$B$};
  \node[above right]at (C) {$C$};
  \node[left]       at (D) {$D$};
\end{tikzpicture}
\caption{The polyhedron \(P\) for the triangle with sides \((a,b,c)=(3,4,5)\).
The equatorial four-cycle \(ABCD\) alternates between heights \(t\) and \(-t\);
hidden edges are dashed.  All eight faces are congruent to the given
triangle.}\label{fig:P}
\end{figure}

\subsection*{The faces are congruent to \(T\)}

Since \(h-t=d\) and \(h+t=s\),
\begin{align*}
  |UA|^{2}&=r^{2}+d^{2}=X+(b^{2}-X)=b^{2},\\
  |UB|^{2}&=r^{2}+s^{2}=X+(c^{2}-X)=c^{2}.
\end{align*}
Moreover,
\begin{align*}
  |AB|^{2}
    &=2r^{2}+(2t)^{2}
     =2X+(s-d)^{2}\\
    &=2X+(c^{2}-X)+(b^{2}-X)-2ds
     =b^{2}+c^{2}-2M=a^{2}.
\end{align*}
Thus \(UAB\) has side lengths \(b,c,a\).

The point configuration \(\{U,V,A,B,C,D\}\) is invariant under the two
isometries
\begin{equation}\label{eq:gens}
  \sigma(x,y,z)=(-y,\,x,\,-z),
  \qquad
  \rho(x,y,z)=(y,\,x,\,-z).
\end{equation}
Indeed \(\sigma\) acts on it by \(U\mapsto V\mapsto U\) and
\(A\mapsto B\mapsto C\mapsto D\mapsto A\), while \(\rho\) acts by
\(U\leftrightarrow V\), \(A\leftrightarrow B\), \(C\leftrightarrow D\).  Here
\(\sigma\) is the improper rotation \(\bar4\) about the \(z\)-axis and \(\rho\)
is a half-turn about the line \(y=x\) in the plane \(z=0\); together they
generate a group of order eight, the point group \(D_{2d}=\bar4 2m\)
\cite{IUCr}.  The \(\sigma\)-orbit of \(UAB\) is \(\{UAB,VBC,UCD,VDA\}\) and
the \(\sigma\)-orbit of \(\rho(UAB)=VAB\) is \(\{VAB,UBC,VCD,UDA\}\), so the
group acts transitively on the eight triangles in \eqref{eq:facets}.  All of
them are therefore congruent to \(T\).

\subsection*{They are exactly the facets}

For the plane through \(U,A,B\), take the normal
\[
  n=(A-U)\times(B-U)=(r,0,-d)\times(0,r,-s)=(dr,\,sr,\,r^{2}).
\]
Since \(r>0\) we have \(n\neq0\), so \(U,A,B\) are affinely independent.  The
remaining three points satisfy
\begin{align*}
  n\mathbin{\cdot}(V-U)&=-(s+d)r^{2}<0,\\
  n\mathbin{\cdot}(C-U)&=-2dr^{2}<0,\\
  n\mathbin{\cdot}(D-U)&=-2sr^{2}<0,
\end{align*}
all strictly negative.  In particular \(V\notin\operatorname{aff}\{U,A,B\}\),
so \(P\) is three-dimensional.  Moreover the plane
\(\{x:n\cdot(x-U)=0\}\) supports \(P\) and meets \(\{U,V,A,B,C,D\}\) in
\emph{exactly} \(\{U,A,B\}\); hence the corresponding facet of \(P\) is the
triangle \(UAB\) itself, and not some larger polygon obtained by merging
coplanar triangles.  Applying the symmetries \eqref{eq:gens} gives the same
conclusion for all eight triangles in \eqref{eq:facets}, and these eight
triangles are pairwise distinct.

Finally, \(P\) has at most six vertices, and a convex three-polytope with
\(v\) vertices has at most \(2v-4\) facets.  Hence \(P\) has at most
\(2\cdot6-4=8\) facets.  Since we have exhibited eight, these are all of them,
and all six points are indeed vertices.  This proves
Theorem~\ref{thm:universal}. \qed

\begin{remark}[Reflections are genuinely needed]\label{rem:chirality}
Let \(T\) be scalene.  The reflection \(\mu(x,y,z)=(-x,y,z)\) is a symmetry of
the configuration, fixing \(U,B,D\) and interchanging \(A\) and \(C\); it
carries the facet \(UAB\) to the facet \(UBC\).  Being orientation-reversing,
it reverses the cyclic order in which the side lengths \(a,b,c\) are read off
around a facet with respect to the outward normal.  Since \(a,b,c\) are
distinct, a cyclic word in them differs from its reverse, so \(UAB\) and
\(UBC\) carry opposite handedness.  The subgroup of orientation-preserving
symmetries, \(\{1,\sigma^{2},\rho,\sigma^{2}\rho\}\), has the two orbits
\(\{UAB,UCD,VAB,VCD\}\) and \(\{UBC,UDA,VBC,VDA\}\); the eight facets thus
split into four direct and four mirror-image copies of \(T\).  The
construction therefore does not bear on the orientation-preserving version of
Problem B22.
\end{remark}

\begin{remark}[Symmetric specializations]
If \(b=c\) then \(d=s\) and \(t=0\), and \(P\) is a square bipyramid over the
equatorial square \(ABCD\).  If \(a=b=c\) then \(X=\tfrac12 a^{2}\) and
\(r=h=a/\sqrt2\), and \(P\) is the regular octahedron.
\end{remark}

\section{The minimum face count}

Throughout this section \(Q\) denotes a convex polyhedron with \(F\) faces,
each congruent to a given nondegenerate triangle \(T\).

\subsection*{Parity}

Suppose \(T\) is not equilateral, and let \(\beta\) be a side length occurring
exactly once among the three sides of \(T\); let \(E_{\beta}\) be the number of
edges of \(Q\) of length \(\beta\).  Every face of \(Q\) is a triangle, and in
a convex polytope no vertex lies in the relative interior of a side of a
facet, so each side of each face is a single edge of \(Q\).  Each face
therefore has exactly one side of length \(\beta\).  Counting incidences
between faces and edges of length \(\beta\) in two ways, and using that every
edge of a convex polyhedron lies on exactly two faces, gives
\begin{equation}\label{eq:parity}
  F=2E_{\beta}.
\end{equation}
Hence \(F\) is even whenever \(T\) is not equilateral.

\needspace{6\baselineskip}
\begin{lemma}[The four-face case]\label{lem:tetrahedron}
A triangle is the common face of a tetrahedron if and only if it is acute.
\end{lemma}

\begin{proof}
Suppose all four faces of a tetrahedron are congruent to \(T\).  Let
\(\ell_{1},\dots,\ell_{k}\) be the distinct side lengths of \(T\), occurring
with multiplicities \(m_{1},\dots,m_{k}\), where \(k\in\{1,2,3\}\) and
\(\sum_{i}m_{i}=3\).  Each face has exactly \(m_{i}\) sides of length
\(\ell_{i}\), so counting incidences as above gives \(4m_{i}=2E_{\ell_{i}}\),
that is, \(E_{\ell_{i}}=2m_{i}\).

We claim that opposite edges are equal.  Note first that two edges of a
tetrahedron sharing a vertex lie in a common face.

If \(k=1\), all six edges are equal and the claim is trivial.  If \(k=2\), say
\(T\) has side lengths \((\lambda,\lambda,\beta)\) with \(\beta\neq\lambda\),
then \(E_{\beta}=2\); the two edges of length \(\beta\) cannot share a vertex,
since the common face would then have two sides of length \(\beta\) while
\(T\) has only one, so they are opposite, and the remaining four edges, all of
length \(\lambda\), constitute the other two opposite pairs.  If \(k=3\), then
\(E_{\ell_{i}}=2\) for each \(i\) and no two equal edges are adjacent, so each
of the three equal pairs is a pair of opposite edges.

The tetrahedron is therefore a disphenoid, and may be written, up to rigid
motion, with vertices
\[
  (x,y,z),\quad (x,-y,-z),\quad
  (-x,y,-z),\quad (-x,-y,z),
\]
these being alternate vertices of the box \([-x,x]\times[-y,y]\times[-z,z]\).
Its three pairs of opposite edges have squared lengths \(4(y^{2}+z^{2})\),
\(4(x^{2}+z^{2})\), and \(4(x^{2}+y^{2})\).  Setting these equal to
\(a^{2},b^{2},c^{2}\), the squared side lengths of \(T\), and solving gives
\begin{equation}\label{eq:disphenoid}
  x^{2}=\frac{b^{2}+c^{2}-a^{2}}{8},\qquad
  y^{2}=\frac{c^{2}+a^{2}-b^{2}}{8},\qquad
  z^{2}=\frac{a^{2}+b^{2}-c^{2}}{8}.
\end{equation}
The four points are affinely independent exactly when \(x,y,z\) are all
nonzero, that is, when all three right-hand sides are positive; and that is
precisely the condition that \(T\) be acute.

Conversely, if \(T\) is acute then \eqref{eq:disphenoid} defines a genuine
tetrahedron, and each of its faces contains exactly one edge from each
opposite pair, hence has side lengths \(a,b,c\).
\end{proof}

Lemma~\ref{lem:tetrahedron} is the classical characterization of the isosceles
tetrahedron, or disphenoid; see Leech \cite{Leech}.

\begin{lemma}[The six-face case]\label{lem:six}
Let \(T\) be the common face of a convex polyhedron with six faces.  Then
either \(T\) is acute, or \(T\) is isosceles with side lengths
\((\lambda,\lambda,\beta)\), \(\beta\neq\lambda\), and \(\beta<\lambda\sqrt3\).
Conversely, every isosceles triangle with \(\beta<\lambda\sqrt3\) is the common
face of a convex polyhedron with six faces.
\end{lemma}

\begin{proof}
If \(T\) is equilateral it is acute and there is nothing to prove, so assume
not, and let \(\beta\) be the side length of \(T\) occurring exactly once.

Since every face of \(Q\) is a triangle, \(3F=2e\), where \(e\) is the number
of edges; with \(F=6\) this gives \(e=9\), and Euler's formula gives \(v=5\)
vertices.  The
only simplicial convex three-polytope with five vertices is the triangular
bipyramid; denote its apices by \(N,S\) and its equatorial vertices by
\(A,B,C\).

Mark the edges of \(Q\) of length \(\beta\).  The dual graph of the bipyramid
is the triangular prism: its vertices are the six faces of \(Q\) and its edges
are the nine edges of \(Q\), the three equatorial edges of \(Q\) corresponding
to the three prism edges that join the two triangles, and the six lateral
edges of \(Q\) to the edges lying within those triangles.  Each face of \(Q\)
contains exactly one marked edge, so the marked edges form a perfect matching
of the prism.  The prism has exactly four perfect matchings, falling into two
orbits of sizes one and three under its automorphism group:
\begin{enumerate}
\item[(I)] the three edges joining the two triangles---that is, the three
  equatorial edges \(AB,BC,CA\) of \(Q\);
\item[(II)] one such edge together with one edge from each triangle---after
  relabelling, the edges \(AB\), \(NC\), and \(SC\) of \(Q\).
\end{enumerate}

\emph{Case} (I).  Here \(|AB|=|BC|=|CA|=\beta\), and every lateral edge carries
one of the two remaining side lengths \(m,m'\) of \(T\).  The faces \(NAB\),
\(NBC\), \(NCA\) force
\[
  \{|NA|,|NB|\}=\{|NB|,|NC|\}=\{|NC|,|NA|\}=\{m,m'\}.
\]
If \(m\neq m'\) this is a proper two-colouring of a triangle, which is
impossible; hence \(m=m'=\lambda\) and \(T\) has side lengths
\((\lambda,\lambda,\beta)\).  All six lateral edges then have length
\(\lambda\), so both apices are equidistant from \(A,B,C\) and lie on the axis
of the equilateral triangle \(ABC\), whose side is \(\beta\) and whose
circumradius is \(\beta/\sqrt3\).  An apex lies off the plane of \(ABC\)
precisely when \(\lambda>\beta/\sqrt3\), that is, when \(\beta<\lambda\sqrt3\).

\emph{Case} (II).  The face \(NAB\), with sides \(|NA|,|NB|,|AB|=\beta\), gives
\(\{|NA|,|NB|\}=\{m,m'\}\) as multisets.  The face \(NBC\), with sides
\(|NB|,|NC|=\beta,|BC|\), gives \(\{|NB|,|BC|\}=\{m,m'\}\), whence
\(|BC|=|NA|\); similarly the face \(NCA\) gives \(|CA|=|NB|\).  The triangle
\(ABC\) therefore has side lengths \(\beta,|NA|,|NB|\) and is congruent to
\(T\).  Since \(N\) does not lie in the plane of \(ABC\), the tetrahedron
\(NABC\) is nondegenerate, and all four of its faces---\(NAB\), \(NBC\),
\(NCA\), and \(ABC\)---are congruent to \(T\).  By
Lemma~\ref{lem:tetrahedron}, \(T\) is acute.

For the converse, let \(T\) have side lengths \((\lambda,\lambda,\beta)\) with
\(\beta<\lambda\sqrt3\).  Place an equilateral triangle \(ABC\) of side
\(\beta\) in the plane \(z=0\), centred at the origin, so that its circumradius
is \(R=\beta/\sqrt3<\lambda\), and set \(N=(0,0,\eta)\), \(S=(0,0,-\eta)\) with
\(\eta=\sqrt{\lambda^{2}-R^{2}}>0\).  The segment \(NS\) meets the interior of
\(ABC\), so \(\operatorname{conv}\{N,S,A,B,C\}\) is a triangular bipyramid with
five vertices and six faces, each of side lengths
\(\lambda,\lambda,\beta\).
\end{proof}

\begin{proof}[Proof of Theorem~\ref{thm:minimum}]
If \(T\) is acute, Lemma~\ref{lem:tetrahedron} realizes it on a tetrahedron,
and no polyhedron has fewer than four faces; hence \(m(T)=4\).

Suppose \(T\) is right or obtuse.  Then \(T\) is not equilateral, so \(F\) is
even by \eqref{eq:parity}, and \(F\neq4\) by Lemma~\ref{lem:tetrahedron};
hence \(F\geq6\).

If \(T\) is isosceles with side lengths \((\lambda,\lambda,\beta)\), then
\(\beta\geq\lambda\sqrt2\), since \(T\) is not acute.  When
\(\beta<\lambda\sqrt3\), the converse part of Lemma~\ref{lem:six} attains
\(F=6\), so \(m(T)=6\).  When \(\beta\geq\lambda\sqrt3\), the direct part of
Lemma~\ref{lem:six} rules out \(F=6\), since \(T\) is neither acute nor
isosceles with \(\beta<\lambda\sqrt3\); parity then gives \(F\geq8\), and
Theorem~\ref{thm:universal} attains \(F=8\), so \(m(T)=8\).

If instead \(T\) is scalene, Lemma~\ref{lem:six} again rules out \(F=6\), and
the same argument gives \(m(T)=8\).
\end{proof}

\begin{remark}
The two ends of the middle range are worth recording.  The right isosceles
triangle \((1,1,\sqrt2)\) has \(m(T)=6\), realized by the bipyramid over an
equilateral triangle of side \(\sqrt2\) with apices at height \(1/\sqrt3\).  At
the other end \(\beta=\lambda\sqrt3\) gives \(\eta=0\): the two apices collapse
onto the centre of \(ABC\), the bipyramid degenerates, and \(m(T)\) jumps to
eight.  The bound \(\beta<\lambda\sqrt3\) is therefore strict.
\end{remark}

\section{Concluding remarks}

The polyhedron \(P\) has the combinatorics of an octahedron: \(U\) and \(V\)
are opposite vertices and \(A,B,C,D\) form an equatorial four-cycle.  With its
\(D_{2d}\) symmetry it is a tetragonal scalenohedron \cite{IUCr}, a shape long
familiar in crystallography, and no priority is claimed for the family itself.
The contribution intended here is the observation that the family is
\emph{universal}---its two shape parameters can be solved in closed form for
an arbitrary prescribed triangle---together with the exact minimum face count
of Theorem~\ref{thm:minimum}.

Monohedral convex polyhedra with isosceles faces have been studied in their
own right; see Malkevitch \cite{Malkevitch} and Eppstein \cite{Eppstein}, the
latter surveying the known infinite families and settling several questions
about realizability with isosceles triangles.  The bipyramid appearing in
Lemma~\ref{lem:six} belongs to the most elementary of those families.

For the triangular case, the principal remaining question is the
orientation-preserving version, in which reflected copies are forbidden.  By
Remark~\ref{rem:chirality} the construction of Section~\ref{sec:eight} uses
four direct and four mirror-image copies of a scalene triangle, so it cannot
address that version.

For acute triangles, however, the orientation-preserving version is already
settled by Lemma~\ref{lem:tetrahedron}.  The half-turns about the three
coordinate axes are symmetries of the disphenoid
\eqref{eq:disphenoid}, and they form a group isomorphic to
\(\mathbb{Z}_{2}\times\mathbb{Z}_{2}\) acting simply transitively on its four
faces.  Being rotations, they preserve outward normals, so all four faces
induce the same cyclic order of side lengths with respect to the outward
normal: no reflected copy is used.  Hence the orientation-preserving minimum
for an acute triangle is also four.

Moreover, for an isosceles triangle the distinction between direct and
reflected copies disappears: such a triangle has a reflection symmetry
\(\tau\), so composing any orientation-reversing congruence with \(\tau\)
yields an orientation-preserving congruence with the same image.  The
constructions above therefore settle the orientation-preserving problem for
every isosceles triangle as well, including the right and obtuse ones.  Among
triangles, then, the orientation-preserving question remains open only for
right and obtuse \emph{scalene} triangles, and by Remark~\ref{rem:chirality}
the eight-face construction does not settle those.

Several further parts of Problem B22 are also left
open here: the classification of \emph{all} attainable face counts, as opposed
to the minimum; the description of the possible arrangements; and the
analogous problem for congruent \(n\)-gons.

\section*{Acknowledgments}

The author used OpenAI’s ChatGPT 5.6 Sol to assist with the initial discovery of the result and drafting of the proof, and Anthropic’s Claude Opus 5 for subsequent review and refinement. The author independently verified the final arguments and assumes responsibility for their correctness.

\end{document}